\documentclass[12pt]{amsart}
\usepackage{xcolor} 
\usepackage{graphicx} 
\usepackage{enumerate}

\usepackage{amsmath}
\usepackage{amssymb}
\usepackage{amsrefs}

\DeclareMathSymbol\HH 0{AMSb}{`H}
\DeclareMathSymbol\I  0{AMSb}{`I}
\DeclareMathSymbol\R  0{AMSb}{`R}

\def\meso{\rlap{\vrule height -1pt width 5pt depth 2pt}\mathbf{P}}

\theoremstyle{plain}
\newtheorem{theorem}{Theorem}[section]
\newtheorem{lemma}[theorem]{Lemma}
\newtheorem{proposition}[theorem]{Proposition}
\newtheorem{corollary}[theorem]{Corollary}
\newtheorem{definition}[theorem]{Definition}

\theoremstyle{definition}

\newtheorem{claim}{Claim}

\theoremstyle{remark}

\newtheorem{fact}{Fact}

\numberwithin{equation}{section}

\DeclareMathOperator{\cf}{cf}
\DeclareMathOperator{\dom}{dom}

\DeclareMathOperator{\sakne}{stem}

\newcommand{\cprec}{\mathop{<\!\!{\circ}}}

\title{On the cofinality of the splitting number}

\author[A. Dow]{Alan Dow}
\address{Department of Mathematics,
University of North Carolina at Charlotte, 
Charlotte, NC 28223}
\email{adow@uncc.edu}
\author[S. Shelah]{Saharon Shelah}
\address{Department of Mathematics, Rutgers University, Hill Center,
 Piscataway, 
 New Jersey, U.S.A. 08854-8019}
\curraddr{Institute of Mathematics\\Hebrew University\\
Givat Ram, Jerusalem 91904, Israel}
\email{shelah@math.rutgers.edu}

\date{\today}

\thanks{The research of the first author was supported by
 the NSF grant No. NSF-DMS 1501506.
The research of the second  author was
 supported by
   the United States-Israel Binational Science Foundation 
(BSF Grant
   no. 2010405), and by the 
   NSF grant No. NSF-DMS 1101597.
 Paper no. 1127 on Shelah's publication list.}

\keywords{
splitting number,  cardinal invariants of the continuum, matrix forcing}

\subjclass{ 03E15 }

\begin{document}
\begin{abstract}
The splitting number $\mathfrak s$ can be singular.
The key method is to construct a forcing poset with finite support
matrix iterations of ccc posets introduced by Blass and the second author
[\textit{Ultrafilters with small generating sets}, {Israel J. Math.},
 \textbf{65},
 {(1989)}]
\end{abstract}
\maketitle

\bibliographystyle{plain}

\section{Introduction}

The cardinal invariants of the continuum discussed in this
article are very well known (see 
\cite{vDHandbook}*{van Douwen, p111})
so we just give a brief reminder.  They deal with the mod finite 
ordering of the infinite subsets of the integers. 
A set $S\subset \omega$
 is \textit{unsplit} by a family $\mathcal Y\subset [\omega]^{\aleph_0}$ 
if $S$ is mod finite contained in one member of $\{ Y, \omega\setminus
Y\}$ for each $Y\in \mathcal Y$. 
The splitting number
 $\mathfrak s$ is the minimum cardinal of a family $\mathcal Y$ for
 which there is no infinite set unsplit by $\mathcal Y$ (equivalently 
every  $S\in [\omega]^{\aleph_0}$  is \textit{split} by some member of
$\mathcal Y$). It is mentioned in \cite{BrendleSingular} that it is
currently unknown if $\mathfrak s$ can be a singular cardinal.

\begin{proposition}
The cofinality of the splitting number is not countable.
\end{proposition}

\begin{proof}
Assume that $\theta$ is the supremum of $\{\kappa_n : n\in \omega\}$
and that there is no splitting family of cardinality less
than $\theta$. Let 
$\mathcal Y = \{ Y_\alpha : \alpha < \theta\}$ be a
family of subsets of $\omega$. Let $S_0 = \omega$ and by induction on $n$,
 choose an infinite subset $S_{n+1}$ of $S_n$ so that
 $S_{n+1}$ is not split by the family $\{ Y_\alpha : \alpha <\kappa_n\}$. 
If $S$ is any pseudointersection of $\{ S_n : n\in \omega\}$, then 
 $S$ is not split by any member of $\mathcal Y$.
\end{proof}

One can easily generalize the previous result and proof to show that 
the cofinality of the splitting number is at least $\mathfrak t$. In
this paper  we prove the following.

\begin{theorem} If $\kappa$ is any uncountable\label{mainth}
 regular cardinal, 
 then there is a $\lambda>\kappa$ with $\cf(\lambda) = \kappa$ and a 
ccc forcing $\mathbb P$ satisfying that $\mathfrak s = \lambda$
in the forcing extension.
\end{theorem}

To prove the theorem,  we construct $\mathbb P$
 using matrix iterations.  
 
 \section{A special splitting family}
 
\begin{definition} Let us say that a family
  $\{ x_i  : i\in I\} \subset [\omega]^\omega$ is $\theta$-Luzin (for an
uncountable
  cardinal $\theta$) if for each $J\in [I]^\theta$, 
 $\bigcap \{x_i : i\in J\}$ is finite
   and $\bigcup \{ x_i : i\in J\}$ is cofinite.
   \end{definition}
   
Clearly a family is $\theta$-Luzin if every $\theta$-sized subfamily is
 $\theta$-Luzin.
   We leave to the reader the easy verification that for a regular
   uncountable cardinal $\theta$, each $\theta$-Luzin family
   is a splitting family. A poset being $\theta$-Luzin preserving will
   have the obvious meaning. For example, any poset of cardinality
   less than a regular cardinal 
 $\theta$ is $\theta$-Luzin preserving.
   
\begin{lemma}
 If $\theta$ is  a regular uncountable \label{luzin} cardinal then 
any   ccc finite
 support iteration of $\theta$-Luzin preserving posets
is again $\theta$-Luzin preserving.
\end{lemma}

\begin{proof}
 We prove this by induction on the length of the iteration. 
 Fix any $\theta$-Luzin family $\{ x_i : i\in I\}$ and let 
  $\langle\langle
   \mathbb P_\alpha :\alpha\leq \gamma\rangle, \langle
      \dot {\mathbb Q}_\alpha : \alpha < \gamma\rangle\rangle$
      be a finite support iteration of ccc posets satisfying
      that $\mathbb P_\alpha$ forces that 
$\dot {\mathbb Q_\alpha}$ is ccc and $\theta$-Luzin
preserving,
      for all $\alpha<\gamma$. 

     If $\gamma$ is a successor ordinal $\beta+1$, 
       then for any  $\mathbb P_\beta$-generic filter $G_\beta$,
the family $\{ x_i : i\in I\}$ is a $\theta$-Luzin family
 in $V[G_\beta]$. 
By the hypothesis on $\dot {\mathbb Q}_\beta$, this family remains
$\theta$-Luzin after further forcing by $\dot {\mathbb Q}_\beta$.

      Now we assume that $\alpha$ is a limit. 
Let $\dot J_0$ be any 
       $\mathbb P_\gamma$-name of a subset of $I$
       and assume that $p\in \mathbb P_{\gamma}$ forces
       that $|\dot J_0| =\theta$.  We must produce a
$q<p$ that forces that $\dot J_0$ is
as in the definition of $\theta$-Luzin.
There is a set $J_1\subset I$
       of cardinality $\theta$ satisfying that, for each $i\in J_1$,
       there is a $p_i<p$ with $p_i\Vdash i\in \dot J_0$. 
The case when
      the cofinality of $\alpha$ not equal to 
 $\theta$ is 
almost immediate. There is a $\beta <\alpha$ such that
 $J_2 = \{ i \in J_1 : p_i\in \mathbb P_\beta\}$ has cardinality
 $\theta$.  
There is a $\mathbb P_\beta$-generic filter $G_\beta$ such
that $J_3 = \{ i \in J_2 : p_i\in G_\beta\}$ has cardinality 
 $\theta$. By the induction hypothesis, the family
 $\{ x_i : i\in I\}$ is $\theta$-Luzin in $V[G_\beta]$ and
 so  we have that 
 $\bigcap \{ x_i : i\in J_3\}$ is finite
 and $\bigcup \{ x_i : i\in J_3\}$ is co-finite. Choose
any $q<p$ in $G_\beta$ and a name $\dot J_3$ for $J_3$
so that $q$ forces this property for $\dot J_3$.
Since $q$ forces that $\dot J_3\subset \dot J_0$, we have
that $q$ forces the same property for $\dot J_0$.

  Finally  we assume that $\alpha $ has cofinality $\theta$. 
 Naturally
      we may assume that the collection $\{ \dom(p_i) : i \in J_1\}$ 
  forms a $\Delta$-system with root contained in some   $\beta < \alpha$.
Again, we may choose   a 
  $\mathbb P_\beta$-generic filter $G_\beta$ satisfying
   that $J_2 = \{ i\in J_1 : p_i\restriction\beta \in G_\beta \}$ has 
   cardinality $\theta$. 
In $V[G_\beta]$, let $\{ J_{2,\xi} : \xi  \in \omega_1\}$ be a
   partition of $J_2$ into pieces of size $\theta$.
   For each $\xi\in \omega_1$, 
   apply  the induction hypothesis
   in the model $V[G_\beta]$, and so we have
that $\bigcap\{ x_i : i\in J_{2,\xi}\}$ is finite
   and $\bigcup \{x_i : i\in J_{2,\xi}\}$ is co-finite. 
For each $\xi\in \omega_1$ let $m_\xi$ be an integer large enough
so that $\bigcap\{x_i: i\in J_{2,\xi}\}\subset m_\xi$
and $\bigcup \{ x_i  : i \in J_{2,\xi}\}\supset \omega\setminus m_\xi$.
Let $m$ be any integer such that $m_\xi = m$ for uncountably many $\xi$.
   Choose any condition $\bar p\in \mathbb P_{\alpha}$ so that
    $\bar p\restriction \beta\in G_\beta$. We prove
    that for each $n>m$ there is a $\bar p_n<\bar p$ so
    that $\bar p_n \Vdash  n\notin \bigcap \{ x_i : i \in \dot I\}$
    and $\bar p_n \Vdash n\in \bigcup \{ x_i : i\in \dot I\}$. 
    Choose any $\xi\in \omega_1$ so that $m_\xi = m$
    and $\dom(p_i) \cap \dom(\bar p)  \subset \beta$ 
    for all $i\in J_{2,\xi}$.
    Now choose any $i_0\in J_{2,\xi}$ so that $n\notin x_{i_0}$. 
   Next choose a distinct $\xi'$ with $m_{\xi'}=m$ so that
     $\dom(p_i)\cap (\dom(\bar p)\cup \dom(p_{i_0}))\subset \beta$
     for all $i\in J_{2,\xi'}$. Now choose $i_1\in J_{2,\xi'}$
     so that $n\in x_{i_1}$. We now have that 
        $\bar p \cup p_{i_0}\cup p_{i_1}$ is a condition
        that forces $\{i_0,i_1\}\subset \dot I$. 
\end{proof}

Next we introduce a $\sigma$-centered poset that will render a given family
non-splitting.

\begin{definition}
 For a  filter $\mathfrak D$ on $\omega$, we define\label{deflaver}
 the Laver style poset
  $\mathbb L(\mathfrak D)$ to be the set of trees $T\subset \omega^{<\omega}$
  with the property that $T$ has a minimal branching node   $\sakne
(T)$
  and for all $\sakne(T)\subseteq t\in T$, the branching set
   $\{ k : t^\frown k \in T\}$ is an element of $\mathfrak D$.
 If $\mathfrak D$ is a filter base for a filter ${\mathfrak D}^*$,
 then $\mathbb L(\mathfrak D)$ will also denote $\mathbb L({\mathfrak
   D}^*)$. 

The name $\dot L = \{ (k, T) : (\exists t)~~t^\frown k \subset 
\sakne(T)\}$ will be referred to as the canonical name for the real
added by $\mathbb L(\mathfrak D)$.
\end{definition}

If $\mathfrak D$ is a principal (fixed) ultrafilter on $\omega$,
 then $\mathbb L({\mathfrak D})$
has a minimum element and so is forcing isomorphic to the trivial
poset.
If $\mathfrak D$ is principal but not an ultrafilter, then
$\mathbb L(\mathfrak D)$ is isomorphic to Cohen forcing. 
If $\mathfrak D$ is a free filter, then $\mathbb L(\mathfrak D)$
adds a dominating real and has similarities to Hechler forcing.
As usual, for a filter (or filter base)
 $\mathfrak D$ of subsets of $\omega$, we use
$\mathfrak D^+$ to denote the set of all subsets of $\omega$
that meet every member of $\mathfrak D$.

\begin{definition}
If $E$ is a dense subset of $\mathbb L(\mathfrak D)$, then 
 a function $\rho_E $ from $\omega^{<\omega}$ into $\omega_1$ 
is a rank function for $E$ if $\rho_E(t) = 0$ 
if and only if $t = \sakne
(T)$ for some $T\in E$,
and
 for all $t\in \omega^{<\omega}$  and $0<\alpha\in \omega_1$, 
 $\rho_E(t) \leq \alpha$ providing the set\label{rank} 
 $\{ k \in \omega : \rho_E(t^\frown k)<\alpha\}$ is in 
 $\mathfrak D^+$. 
\end{definition}

When $\mathfrak D$ is a free filter, then 
 $\mathbb L(\mathfrak D)$ has cardinality $\mathfrak c$, but
nevertheless, if $\mathfrak D$ has a base of cardinality
less than a regular cardinal $\theta$, 
 $\mathbb L(\mathfrak D)$ is $\theta$-Luzin preserving.

\begin{lemma}
If $\mathfrak D$ is a free filter on $\omega$ and if 
$\mathfrak D$ has a base of\label{laverpreserves}
 cardinality less than a regular uncountable
cardinal $\theta$, then $\mathbb L(\mathfrak D)$ is 
 $\theta$-Luzin preserving.
\end{lemma}

\begin{proof}
Let $\{ x_i : i\in \theta\}$ be a $\theta$-Luzin family
with $\theta$ as in the Lemma. Let $\dot J$ be a
 $\mathbb L(\mathfrak D)$-name of a subset of 
$\theta$. We prove that if
$\bigcap \{ x_i : i\in \dot J\}$ is not finite,
then $\dot J$ is bounded in $\theta$. By symmetry,
 it will also prove that if $\bigcup\{x_i : i\in \dot J\}$
is not cofinite, then $\dot J$ is bounded in $\theta$.
 Let $\dot y$ be the 
 $\mathbb L(\mathfrak D)$-name of the intersection,
and let  $T_0$ be any member of $\mathbb L(\mathfrak D)$
that forces that $\dot y$ is infinite.
Let $M$ be any $<\theta$-sized elementary submodel of 
 $H((2^{\mathfrak c})^+)$ such that $T_0,\mathfrak D$, $\dot J$,
and $\{ x_i :  i\in \theta\}$ are all members of $M$
and such that $M\cap \mathfrak D$ contains a base for $\mathfrak D$.  
Let $i_M = \sup(M\cap \theta)$.  If $x\in M\cap [\omega]^\omega$, 
 then $I_x = \{ i\in \theta : x\subset x_i\}$ is an element
of $M$ and has cardinality less than $\theta$. Therefore,
 if $i\in \theta\setminus i_M$,
 then $x_i$ does not contain any infinite subset
of $\omega$ that is an element of $M$. We prove
that $x_i$ is forced by $T_0$ to also not contain $\dot y$.
 This will prove that $\dot J$ is bounded by $i_M$.
Let $T_1<T_0$ be any condition in $\mathbb L(\mathfrak D)$
and let $t_1 = \sakne
(T_1)$. We show that 
 $T_1$ does not force that $x_i\supset \dot y$.
We define the relation $\Vdash_w$ on $T_0\times \omega$
to be the set
$$ \{ (t,n)\in T_0\times\omega : ~\mbox{there is no}~~T\leq T_0,
\sakne
(T)=t,   \mbox{s.t.}\ T\Vdash n\notin \dot y\}~.$$ 
For convenience we may write, for $T\leq T_0$,
$T\Vdash_w n\in \dot y$ providing $(\sakne
(T),n)$ is in
$\Vdash_w$, and this is equivalent to the relation
that $T$ has no stem  preserving extension 
forcing that $n$ is not in $\dot y$. Let $T_2\in M$
be any extension of $T_0$ with stem  $t_1$. 
 Let $L$ denote the set of $\ell\in \omega$ such that 
 $T_2 \Vdash_w \ell\in \dot y$. 
 If $L$ is infinite,
 then, since $L\in M$,
 there is an $\ell\in L\setminus x_i$. This implies 
that $T_1$ does not force $x_i\supset \dot y$,
since $T_2\Vdash_w j\in \dot y$ implies that
$T_1$ fails to force that $\ell \notin \dot y$.

Therefore we may assume that 
 $L$ is finite  
   and let $\ell$ be the maximum of $L$. 
Define the set $E\subset \mathbb L(\mathfrak D)$ according
 to $T\in E$ providing that either $t_1\notin T$ 
or  there is a $j>\ell$ such that
  $T \Vdash_w j\in \dot y$. Again this set $E$ is in $M$
and is easily seen to be a dense subset of $\mathbb L(\mathfrak D)$. 
By the choice of $\ell$, we note that $\rho_E(t_1) >0$.
If $\rho_E(t_1) > 1$, then
the set $\{ k\in \omega : 0<\rho_E(t_1^\frown k) < \rho_E(t_1)\}$
is in $\mathfrak D^+$ and so there is a $k_1$  in this set
such that $t_1^\frown k_1 \in T_1\cap T_2$. By a finite induction, we
can choose an extension $t_2\supseteq t_1$ so that 
 $t_2 \in T_1\cap T_2$ and $\rho_E(t_2) = 1$. Now, there is a set
 $D\in \mathfrak D\cap M$ contained in 
$ \{ k : t_2^\frown k \in T_1\cap T_2\}$ since $M$ contains a base
for $\mathfrak D$. Also, $D_E = \{ k\in D : 
 \rho_E(t_2^\frown k) = 0\}$ is in $\mathfrak D^+$. 
For each $k\in D_E$, 
 choose the minimal $j_k$ so that $T_2 ^\frown k \Vdash j_k\in \dot
 y$. The set $\{ j_k : k\in D_E\}$ is an element of $M$. This set
is not finite because if it were then there would be a single
 $j$ such that $\{ k\in D_E : j_k = j\}\in \mathcal D^+$, 
which would contradict that $\rho_E(t_2) >0$. This means
that there is a $k\in D_E^+$ with $j_k\notin x_i$, 
 and again we have shown that $T_1$ fails to force that
 $x_i$ contains $\dot y$.
\end{proof}

\section{Matrix Iterations}

The terminology ``matrix iterations'' is used in  
 \cite{BrendleFischer}, see also forthcoming preprint 
(F1222) from the second author.
The paper \cite{BrendleFischer}  nicely expands on the
method  of matrix iterated forcing first introduced
in \cite{BlassShelah}. 

Let us recall that a poset $(P,<_P)$ is a complete suborder of a 
poset $(Q,<_Q)$ providing $P\subset Q$, ${<_P}\subset{<_Q}$,
and each maximal antichain of $(P,<_P)$ is also
a maximal antichain of $(Q,<_Q)$. 
Note that it follows that incomparable members of $(P,<_P)$
are still incomparable in $(Q,<_Q)$, i.e. $p_1  \perp_P p_2$ 
implies $p_1 \perp_Q p_2$.
We use the notation
$(P,<_P) <\!\!{\circ }\  (Q,<_Q)$ to abbreviate the complete suborder
relation,  
and similarly use $P<\!\!{\circ }\  Q$ 
if $<_P$ and $<_Q$ are clear from the context.
 An element $p$ of $P$
is a reduction of $q\in Q$ if $r\not \perp_Q q$ for
each $r<_P p$. If $P\subset Q$, $<_P\subset <_Q$,
$\perp_P~ \subset~ \perp_Q$, and each element of $Q$
has a reduction in $P$, then $P <\!\!{\circ }\ Q$.
The reason is that if $A\subset P$ is a maximal antichain
and $p\in P$ is a reduction of $q\in Q$, then there is
an $a\in A$ 
and an $r$ less than both $p$ and $a$  in $P$, such that $r\not
\perp_Q q$.

\begin{definition}
We will say that an object $\meso$ is a 
matrix iteration if there is an infinite cardinal $\kappa$ 
and an ordinal $\gamma$ (thence a $(\kappa ,\gamma)$-matrix
iteration) 
such that 
  $\meso = 
\langle
\langle {\mathbb P}^{\meso}_{i,\alpha} 
 :  i\leq \kappa  ,
\alpha \leq \gamma\rangle , 
\langle \dot{\mathbb Q}^{\meso}_{i,\alpha} 
 : i\leq \kappa, \alpha < \gamma\rangle \rangle$
where, for each $(i,\alpha)\in \kappa  + 1\times \gamma$
and 
 each $j<i$,
\begin{enumerate}
\item ${\mathbb P}^{\meso}_{j,\alpha}$ is a complete
suborder  of the poset ${\mathbb P}^{\meso}_{i,\alpha}$ (i.e.
  ${\mathbb P}^{\meso}_{j,\alpha} {<\!\!{\circ }}\ {\mathbb
    P}^{\meso}_{i,\alpha}$),
\item  $\dot {\mathbb Q}^{\meso}_{i,\alpha}$ is a 
 ${\mathbb P}^{\meso}_{i,\alpha}$-name of a ccc poset,
 ${\mathbb P}^{\meso}_{i,\alpha+1}$ is equal to ${\mathbb
   P}^{\meso}_{i,\alpha}  
* \dot{\mathbb Q}^{\meso}_{i,\alpha}$, 
\item for limit $\delta\leq \gamma $,
 ${\mathbb P}^{\meso}_{i,\delta }$ is equal to the union of the family
 $\{ {\mathbb P}^{\meso}_{i,\beta} : \beta < \delta\}$ 
\item  ${\mathbb P}^{\meso}_{\kappa,\alpha}$ is the union of the chain
 $\{ {\mathbb P}^{\meso}_{j,\alpha} : j< \kappa\}$.
\end{enumerate}
\end{definition}

When the context makes it clear, we omit the superscript
 $\meso$ when discussing a matrix iteration.
Throughout the paper, $\kappa$ will be a fixed 
uncountable regular cardinal

\begin{definition}
A sequence $\vec\lambda$ is $\kappa$-tall if 
  $\vec\lambda = \langle \mu_\xi, \lambda_\xi : \xi <\kappa\rangle$ is
a  sequence of pairs of
 regular cardinals satisfying that $\mu_0=\omega<
\kappa <
  \lambda_0$ and, for $0<\eta<\kappa$,
${\mu_\eta} < \lambda_{\eta}$
where
  $\mu_\eta = (2^{\sup\{\lambda_\xi: \xi<\eta\}})^+$.
\end{definition}

Also for the remainder of the paper, we
 fix a $\kappa$-tall sequence $\vec\lambda $
and  $\lambda$ will denote the supremum of the set 
 $\{ \lambda_\xi :\xi \in \kappa\}$. 
For simpler notation, whenever we discuss a  matrix
iteration $\meso$ we shall henceforth
assume that it is a $(\kappa,\gamma)$-matrix
iteration
for some ordinal $\gamma$.  We may refer to a forcing extension
by $\meso$ as an abbreviation for the forcing extension by
 $\mathbb P^{\meso}_{\kappa,\gamma}$.

\bigskip

For  any poset  $P$,  any $P$-name $\dot D$,
 and $P$-generic filter $G$, $\dot D[G]$ will denote
 the valuation of $\dot D$ by $G$. For any ground model 
 $x$, $\check x$ denotes the canonical name so
that $\check x[G] = x$. When $x$ is an ordinal (or an integer) we will
suppress the accent in $\check x$.
A 
$P$-name $\dot D$ 
of a subset of $\omega$
 will be said to be \textit{nice}\/ or
 \textit{canonical}\/
if
for each integer $j\in \omega$, there is an antichain
 $A_j$ such that $\dot D = \bigcup\{ \{j\}\times A_j : j\in
 \omega\}$.
We will say that $\dot{\mathfrak D}$ is a nice $ P$-name
of a family of subsets of $\omega$
 just to mean that $\dot{\mathfrak D}$ is a collection
 of nice $P$-names of subsets of $\omega$.
 We will use $(\dot{\mathfrak D})_P$ if we need to emphasize
 that we mean the $P$-name. 
Similarly if we say that $\dot{\mathfrak D}$ is a nice $P$-name
 of a filter (base) we mean that $\dot{\mathfrak D}$ is a
 nice $P$-name such that, for each $P$-generic
 filter, the collection $\{ \dot D[G] : \dot D\in 
\dot{\mathfrak D}\}$ is a 
 filter (base) of infinite subsets of $\omega$.

Following these conventions, the following notation will be helpful.

\begin{definition}
  For a $(\kappa,\gamma)$-matrix\label{defineB}
$\meso$ and $i<\kappa $,  we let
${\mathbb B}_{i,\gamma}^{\meso}$ denote the set of all nice
${\mathbb P}_{i ,\gamma}^{\meso}$-names of subsets of $\omega$.
We note that
 this then is the nice ${\mathbb P}_{i,\gamma}^{\meso}$-name
for the power set of $\omega$.
As usual, when possible we suppress the $\meso$ superscript.
\end{definition}

For a  nice $\meso$-name 
$\dot{\mathfrak D}$
of a filter (or filter base) of subsets of $\omega$, we let
$(\dot{\mathfrak D})^+$ denote the set of all nice $\meso$-names
that are forced to meet every member of $\dot{\mathfrak D}$. It follows
that $(\dot{\mathfrak D})^+$ is the nice $\meso$-name
for the usual 
defined notion $(\dot{\mathfrak D})^+$ in the forcing extension
by $\meso$. We 
let $\langle\dot{\mathfrak D}\rangle$ 
denote the nice $\meso$-name of the filter generated by
$\dot{\mathfrak D}$.  We use the same notational conventions if,
for some poset $\mathbb P$, 
$\dot{\mathfrak D}$ is a nice $\mathbb P$-name
of a filter (or filter base) of subsets of $\omega$.

The main idea for controlling  the splitting 
number in the extension by $\meso$ will involve having
many of the subposets being $\theta$-Luzin preserving
for $\theta\in \{\lambda_\xi : \xi\in \kappa\}$. Motivated
by the fact that posets
 of the form  $\mathbb L(\mathfrak D)$ (our proposed
 iterands)
are $\theta$-Luzin preserving when $\mathfrak D$ is
sufficiently small we adopt the name $\vec\lambda$-thin
for this next notion.

\begin{definition}
  For a $\kappa$-tall sequence $\vec\lambda$, 
  we will say that a $(\kappa,\gamma)$-matrix-iteration $\meso$ is 
  $\vec\lambda$-thin
  providing that for each  $\xi<\kappa$ and $\alpha\leq \gamma$,
$\mathbb P^{\meso}_{\xi,\alpha} $  is
$\lambda_\xi$-Luzin preserving.
\end{definition}

Now we combine the notion of $\vec\lambda$-thin matrix-iteration
with Lemma \ref{luzin}.  We adopt Kunen's  notation that for a set
 $I$, $\operatorname{Fn}(I,2)$ denotes the usual poset for adding
 Cohen reals (finite partial functions from $I$ into $2$ ordered by
 superset).

\begin{lemma}
 Suppose that  $\meso$ is a $\vec\lambda$-thin $(\kappa,\gamma)$-matrix
 iteration\label{CohenLuzin}
   for some $\kappa$-tall sequence $\vec\lambda$. Further
 suppose that $\dot{\mathbb Q}_{i,0}$ is the $\mathbb P_{i,0}$-name
 of the poset $\operatorname{Fn}(\lambda_\xi,2)$ for each $\xi\in\kappa$,
 and therefore
  $\mathbb P_{\kappa,1}$ is isomorphic to $\operatorname{Fn}(\lambda,2)$.
Let $\dot g $ denote
 the generic function from $\lambda$ onto $2$ added by
    $ {\mathbb P}_{\kappa,1}$ and, for $i<\lambda$,
 let $\dot x_{i}$ be the canonical name of
 the set $ \{ n\in \omega : \dot g(i+n) = 1\}$.  
 Then  the family
   $\{ \dot x_{ i} :    i< \lambda  \}$ is forced by $\meso$
   to be a splitting family.
\end{lemma}

\begin{proof}
 Let $G_{\kappa,\gamma}$ be a $\mathbb P_{\kappa,\gamma}$-generic
 filter.
 For each $\xi\in\kappa$ and $\alpha\leq \gamma$, let
 $G_{\xi,\alpha} = G_{\kappa,\gamma}\cap
  \mathbb P_{\xi,\alpha}$.
 Let $\dot y$ be any nice $\mathbb P_{\kappa,\gamma}$-name
 for a subset of $\omega$.  Since $\dot y$ is a countable name, 
  we may choose a $\xi<\kappa$ so that $\dot y$ is a
  $\mathbb P_{\xi,\gamma}$-name.
 It is easily shown, and very well-known, that
 the family $\{ \dot x_{i} : i <\lambda_\xi\}$ is forced
 by $\mathbb P_{\xi,1}$ (i.e. $\operatorname{Fn}(\lambda_\xi,2)$)
 to be a $\lambda_\xi$-Luzin family.  
By the hypothesis that $\meso$ is $\vec\lambda$-thin,
  we have, by Lemma \ref{luzin},
that $\{ \dot x_{i} : i < \lambda_\xi\}$ is still
   $\lambda_\xi$-Luzin in $V[G\cap \mathbb P_{\xi,\gamma}]$.
Since $\dot y$ is a $\mathbb P_{\xi,\gamma}$-name,
there is an $i<\lambda_\xi$ such that $ 
\dot y[G_{\xi,\gamma}]\cap \dot x_i[G_{\xi,\gamma}]$ and
 $\dot y[G_{\xi,\gamma}]
\setminus \dot x_i[G_{\xi,\gamma}]$ are infinite.
\end{proof}

\section{The construction of $\meso$}

When constructing a matrix-iteration by recursion, we will
need  notation and language for extension.
We will use, for an ordinal $\gamma$, $\meso^{\gamma}$ to indicate
that $\meso^\gamma$ is a
$(\kappa,\gamma)$-matrix iteration.

\begin{definition}  \begin{enumerate}
\item  A\label{extension}
matrix iteration $\meso^\gamma$ is an extension
  of $\meso^{\delta}$ providing $\delta\leq \gamma$, and, for each
  $\alpha\leq \delta$ and $i\leq \kappa$,
 ${\mathbb P}^{\meso^\delta}_{i,\alpha} = 
{\mathbb P}^{\meso^\gamma}_{i,\alpha} $. We can use
 $\meso^\gamma\restriction \delta$ to denote the unique
$(\kappa  ,\delta)$-matrix iteration extended by
 $\meso^\gamma$.

\item If, for each $i <\kappa$, $\dot {\mathbb Q}_{i,\gamma }$ is a 
  ${\mathbb P}^{\meso}_{i,\gamma }$-name of a ccc poset
  satisfying that, for each $i<j < \kappa $, $
  \mathbb P_{i,\gamma}*\dot{\mathbb Q}_{i, \gamma}$ is a complete
  subposet of $\mathbb P_{j,\gamma}*\dot{ \mathbb Q}_{j,\gamma }$,
  then we let $\meso{} * \langle \dot {\mathbb Q}_{i,\gamma}
   : i < \kappa
  \rangle$ denote the $(\kappa,\gamma+1)$-matrix  
  $\langle \langle \mathbb P_{i,\alpha} : i\leq\kappa,
\alpha \leq \gamma+1\rangle,
   \langle \dot {\mathbb Q}_{i,\alpha} : i\leq \kappa,
\alpha <\gamma+1\rangle 
\rangle$,  where  $\dot {\mathbb Q}_{\kappa,\gamma}$ is the
     $\meso$-name of the union of $\{ \dot {\mathbb Q}_{i,\gamma} :
i<\kappa\}$  and, for $i\leq \kappa $,
  $\mathbb P_{i,\gamma } = 
    {\mathbb P}^{\meso}_{i,\gamma}$, 
     $\mathbb P_{i,\gamma+1 } = 
    {\mathbb P}^{\meso}_{i,\gamma } * \dot {\mathbb Q}_{i,\gamma }$,
    and for $\alpha <\gamma$, 
     $(\mathbb P_{i,\alpha}, \dot {\mathbb Q}_{i,\alpha}) =
       ({\mathbb P}^{\meso}_{i,\alpha}, \dot {\mathbb
      Q}^{\meso}_{i,\alpha})$.
  \end{enumerate}
\end{definition}

The following, from \cite{BrendleFischer}*{Lemma 3.10},
shows that extension at limit steps
 is canonical.

\begin{lemma} If $\gamma $ is a limit\label{limitMatrix}
and if $\{\meso^\delta : \delta <\gamma\}$ is a sequence
of matrix iterations satisfying that for $\beta < \delta <
\gamma$,  $\meso^\delta\restriction \beta = \meso^\beta$, 
then there is a unique matrix iteration $\meso^\gamma$
such that $\meso^\gamma\restriction \delta = \meso^\delta$
for all $\delta < \gamma$.
\end{lemma}

\begin{proof}
For each $\delta <\gamma$ and $i <\kappa $, we
define ${\mathbb P}^{\meso^\gamma}_{i,\delta } $
 to be $ {\mathbb P}^{\meso^\delta}_{i,\delta }$
and $\dot{\mathbb Q}^{\meso^\gamma}_{i, \delta } $ to be $\dot {\mathbb
  Q}^{\meso^{\delta+1}}_{i,\delta }$. It follows that 
$\dot{\mathbb Q}^{\meso^\gamma}_{i,\delta } $ is a
${\mathbb P}^{\meso^\gamma}_{i,\delta } $-name.
 Since $\gamma$ is a limit,
 the definition of ${\mathbb P}^{\meso^\gamma}_{i,\gamma }$ 
is required to be $\bigcup \{ {\mathbb P}^{\meso^\gamma}_{i,\delta } : 
 \delta <\gamma\}$ for $i<\kappa$. Similarly, 
 the definition of ${\mathbb P}^{\meso^\gamma}_{\kappa,\gamma}$ 
is required to be $\bigcup \{ {\mathbb P}^{\meso^\gamma}_{i,\gamma} : 
 i<\kappa \}$. Let us note that
${\mathbb P}^{\meso^\gamma}_{\kappa  ,\gamma}$  is also required to
be the union of the chain 
$\bigcup \{ {\mathbb P}^{\meso^\gamma}_{\kappa,\delta} : 
 \delta <\gamma\}$, and this holds by assumption
on the sequence $\{\meso^\delta : \delta < \gamma\}$.

 To prove that $\meso^{\gamma}$ is a 
 $(\kappa,\gamma)$-matrix it remains to prove that
for $j<i\leq \kappa$, and 
 each $q\in {\mathbb P}^{\meso^\gamma}_{i,\gamma}$,
 there is a reduction $p$ in ${\mathbb P}^{\meso^\gamma}_{j,\gamma}$.  
Since $\gamma$ is a limit, there is an $\alpha < \gamma$
such that
 $q\in {\mathbb P}^{\meso^\alpha}_{i,\alpha}$ and, by assumption,
 there is a reduction, $p$, of $q$ in 
${\mathbb P}^{\meso^\alpha}_{j,\alpha}$. 
By induction on $\beta$ ($\alpha \leq \beta \leq \gamma$)
we note that $q\in {\mathbb P}^{\meso^\beta}_{i,\beta}$
and that  $p$ is a reduction of $q$ 
in ${\mathbb P}^{\meso^\beta}_{j,\beta}$. For limit $\beta$ it is
trivial, and for successor $\beta$ it follows from condition
  (1) in the definition of  matrix iteration.
\end{proof}

We also will need the next result taken from 
 \cite{BrendleFischer}*{Lemma 13}, 
which 
they describe as well known, for stepping diagonally in the array of
posets. 

\begin{lemma} 
Let $\mathbb P, \mathbb Q$ be partial orders such that
 $\mathbb P$ is a complete suborder of $\mathbb Q$. 
Let $\dot {\mathbb A}$ be a $\mathbb P$-name for a
 forcing notion\label{successorMatrix}
  and let $\dot {\mathbb B}$ be a $\mathbb Q$-name for a forcing
notion such that $\Vdash_{\mathbb Q} \dot {\mathbb A}
\subset 
\dot {\mathbb B}$, and every $\mathbb P$-name of a maximal antichain
of $\dot {\mathbb A}$ is also forced by $\mathbb Q$ to be a maximal
antichain of $\dot {\mathbb B}$. 
 Then 
$\mathbb P * \dot {\mathbb A}
 <\!\!{\circ}~~ \mathbb Q * \dot  {\mathbb B}$  
\end{lemma}

Let us also note if $\dot {\mathbb B} $ is equal to $\dot {\mathbb A}$
in Lemma \ref{successorMatrix},  then the hypothesis and the conclusion
of the Lemma are immediate.  On the other hand, 
if $\dot {\mathbb A}$ is the $\mathbb P$-name of
$\mathbb L(\dot {\mathfrak D})$ for some
 $\mathbb P$-name of a filter $\dot{\mathfrak D}$, then the
$\mathbb Q$-name of $\mathbb L(\dot {\mathfrak D})$ is not
necessarily equal to $\dot {\mathbb A}$.

\begin{lemma}[\cite{charspectrum}*{1.9}]
 Suppose that
 $\mathbb P,\mathbb Q$ are posets
with   $\mathbb P\cprec  \mathbb Q$.
Suppose also that  $\dot{\mathfrak D}_0$ is 
 a $\mathbb P$-name of a  filter on $\omega$ and
 $\dot{\mathfrak D}_1$ is a\label{laver} 
 $\mathbb Q$-name of a  filter on $\omega$.
  If $\Vdash_{\mathbb Q} \dot{\mathfrak D}_0
\subseteq \dot{\mathfrak D}_1 $
then $\mathbb P*\mathbb L(\dot{\mathfrak D}_0)$  is a complete
subposet of $    \mathbb   Q*\mathbb L(\dot{\mathfrak D}_1)$
if either of the two equivalent conditions hold:
\begin{enumerate}
\item $\Vdash_{\mathbb Q} ((\dot{\mathfrak D}_0)^+)_{\mathbb P}  \subseteq
  \dot{\mathfrak D}_1^+$, 
\item $\Vdash_{\mathbb Q}
\dot {\mathfrak D}_1\cap V^{\mathbb P} 
\subseteq \langle\dot{\mathfrak D}_0\rangle$ 
(where $V^{\mathbb P}$ is the class of
$\mathbb P$-names).
\end{enumerate}
\end{lemma}

\begin{proof}
Let $\dot E$ be any $\mathbb P$-name of a maximal antichain
of $\mathbb L(\dot{\mathfrak D}_0)$. By Lemma \ref{successorMatrix},
it suffices to show that $\mathbb Q$ forces that every member
of $\mathbb L(\dot{\mathfrak D}_1)$ is compatible with some member
of  $\dot E$. Let $G$ be any $\mathbb Q$-generic filter
and let $E$ denote the valuation of $\dot E$ by $G\cap \mathbb P$.
Working in the model $V[G\cap \mathbb P]$, we
  have the function $\rho_E $ as in Lemma \ref{rank}.
  Choose   $\delta\in \omega_1$ satisfying
 that $\rho_E(t) < \delta$ for all $t\in \omega^{<\omega}$.
 Now,
 working in $V[G]$, 
 we consider any $T\in \mathbb L(\dot{\mathfrak D}_1)$ and we find
 an element of $E$ that is compatible with $T$.  In fact,  
 by induction on $\alpha <\delta$, one easily proves
 that for each $T\in \mathbb L(\dot{\mathfrak D}_1)$ with
 $\rho_E(\sakne(T)) 
 \leq \alpha$, $T$ is compatible with some member of $E$.
\end{proof}

\begin{definition}
  For  a $(\kappa,\gamma)$-matrix-iteration $\meso$,
and ordinal $i_\gamma < \kappa$, 
  we say that\label{lambdathin}
  an increasing
 sequence $\langle \dot{ \mathfrak D}_i : i <   \kappa \rangle$
  is a 
  $(\meso,\vec\lambda(i_\gamma))$-thin sequence of filter bases,
  if for each $i<j<\kappa$ 
  \begin{enumerate}
\item $\dot{\mathfrak D}_{i}$
is a subset
of $ \mathbb B_{i,\gamma}$ 
(hence a nice ${\mathbb  P}_{i ,\gamma}^{\meso}$-name)
\item $\Vdash_{\mathbb P_{i,\gamma}}\dot{\mathfrak D}_{i}$ 
 is a filter with a base of cardinality at most  $\mu_{i_\gamma}$,
\item
  $\Vdash_{\mathbb P_{j,\gamma}}
\langle\dot {\mathfrak D}_j\rangle\cap \mathbb B_{i,\gamma}
 \subseteq
\langle     \dot{\mathfrak D}_{i } \rangle$.
  \end{enumerate}
\end{definition}

Notice that a $(\meso,\vec\lambda(i_\gamma))$-thin sequence of filter bases
can be (essentially) eventually constant. Thus we will say that a sequence
$\langle \dot {\mathfrak D}_i  :  i  \leq j \rangle$ (for some
$j<\kappa$) 
is 
a $(\meso,\vec\lambda(i_\gamma))$-thin sequence of filter bases if
the sequence
$\langle\dot  {\mathfrak D}_i  : i <
 \kappa \rangle$ is
a $(\meso,\vec\lambda(i_\gamma))$-thin sequence of filter bases
where $\dot {\mathfrak D}_i $
is the $\mathbb P_{i,\gamma}$-name 
for $\mathbb B_{i,\gamma} \cap \langle
\dot{\mathfrak D}_{j}\rangle$
for $j  < i\leq \kappa$.
When $\meso$ is clear from the context, we will use
 $\vec\lambda(i_\gamma)$-thin as an abbreviation for
 $(\meso,\vec\lambda(i_\gamma))$-thin.

\begin{corollary}
  For  a  $(\kappa,\gamma)$-matrix-iteration $\meso$,
ordinal $i_\gamma <\kappa$, 
and  a\label{extL} 
  $(\meso,\vec\lambda(i_\gamma))$-thin sequence of filter bases
$\langle \dot{\mathfrak D}_\xi : 
i <
\kappa\rangle$,  
$\meso*\langle \dot {\mathbb Q}_{i,\gamma} : i\leq \kappa\rangle$
is a $\gamma+1$-extension of $\meso$,
where, for each $i \leq i_\gamma$, $\dot{\mathbb Q}_{i,\gamma}$ is
the trivial poset, and for $i_\gamma\leq i <\kappa $,
$\dot{\mathbb Q}_{i,\gamma}$ is
$\mathbb L(\dot{\mathfrak D}_{i})$.
\end{corollary}

\begin{definition}
Whenever $\langle \dot{\mathfrak D}_i : i<\kappa\rangle$
is a $(\meso,\vec\lambda(i_\gamma))$-thin sequence of filter bases,
  let $\meso*\mathbb L(\langle \dot{\mathfrak D_i} :i_\gamma\leq i<\kappa
  \rangle)$   
denote the $\gamma+1$-extension described in Corollary \ref{extL}.
\end{definition}

This next corollary is immediate.

\begin{corollary}
  If $\meso$ is  a $\vec\lambda$-thin $(\kappa,\gamma)$-matrix
  and if $\langle \dot{\mathfrak D_i} :  i < \kappa
\rangle$ is
  a  $(\meso,\vec\lambda(i_\gamma))$-thin sequence of filter bases,
  then $\meso*\mathbb L(
  \langle \dot{\mathfrak D_i}  :i_\gamma\leq i<\kappa \rangle)$
  is a $\vec\lambda$-thin $(\kappa,\gamma+1)$-matrix.
\end{corollary}

We now describe a first approximation of
the scheme, $\mathcal K(\vec\lambda)$,
 of posets that we will be using to produce
the model.

\begin{definition}
  For an ordinal $\gamma>0$ and a $(\kappa,\gamma)$-matrix iteration
  $\meso$, we will say that 
  $\meso\in \mathcal K(\vec \lambda)$ providing for each
  $0<\alpha<\gamma$,  
  \begin{enumerate}
\item for each $i\leq\kappa$, ${\mathbb P}_{i,1}^{\meso}$ is
  $\operatorname{Fn}(\lambda_i,2)$, and 
\item there is an $i_\alpha = i^{\meso}_\alpha <\kappa$
and a   $({\meso}\restriction
 \alpha,\vec\lambda(i_\alpha) )$-thin 
sequence
$ \langle {\dot{\mathfrak D}}^\alpha_i : i  <\kappa \rangle $ of
filter bases, such that
 $\meso\restriction\alpha+1$
  is equal to    $\meso \restriction \alpha *
  {\mathbb L}(\langle\dot{ \mathfrak D}^\alpha_i : i_\alpha \leq i
  <\kappa\rangle)$.
  \end{enumerate}
For each $0<\alpha<\gamma$,
we let $\dot{\mathfrak D}^\alpha_\kappa$ denote
the $\meso\restriction \alpha$-name of the union
$\bigcup \{\dot{ \mathfrak D}^\alpha_i : i_\alpha\leq i
  <\kappa\}$, and we
 let $\dot L_\alpha$ denote the canonical
$\meso\restriction\alpha+1$-name of the subset of $\omega$ added
by $ \mathbb L(\mathfrak D^\alpha_\kappa)$.
\end{definition}

Let us note that each $\meso\in \mathcal K(\vec \lambda)$ is
$\vec\lambda$-thin. Furthermore, by Lemma \ref{CohenLuzin},
this means that each $\meso\in \mathcal
K(\vec\lambda)$ forces that $\mathfrak s\leq \lambda$.
We begin a new section for the task of proving that there
 is a $\meso\in \mathcal K(\vec\lambda)$ that forces that 
 $s\geq\lambda$.

It will be important to be able to construct 
$(\meso,\vec\lambda(i_\gamma))$-thin sequences of filter bases, and it
seems  we will need some help.

\begin{definition}
  For an ordinal $\gamma>0$ and\label{H}
 a $(\kappa,\gamma)$-matrix iteration
  $\meso$ we will say that $\meso\in \mathcal H(\vec\lambda)$
if $\meso$ is
in $\mathcal K(\vec \lambda)$ and for each
  $0<\alpha<\gamma$,  if $i_\alpha = i^{\meso}_\alpha>0$ then
 $\omega_1\leq \cf(\alpha) \leq \mu_{i_\alpha}$ and
there is a $\beta_\alpha<\alpha$ such that
  \begin{enumerate}
\item for $\beta_\alpha\leq \xi <\alpha$, $i_\xi \in \{0,i_\alpha\}$,
\item if $\beta_\alpha \leq \eta<\alpha$, $i_\eta>0$
and $\xi = \eta+\omega_1\leq\alpha$, then 
 $\dot L_\eta\in \dot{\mathcal D}^{\xi}_{i_\xi}$, and 
 ${\mathbb P}_{i_\xi,\xi}\Vdash \dot {\mathcal D}^\alpha_{i_\xi}$
has a descending mod finite base of cardinality $\omega_1$,
\item if $\beta_\alpha < \xi \leq \alpha$, $i_\xi>0$, 
and $\eta+\omega_1 <\xi$ for $\eta<\xi$, then
 $\{ \dot L_{\eta} : \beta_\alpha \leq \eta <\alpha, \cf(\eta)
 \geq\omega_1\}$ is a base for $\dot{\mathcal D}^\xi_{i_\xi}$.
  \end{enumerate}
\end{definition}

\section{Producing $\vec\lambda$-thin filter sequences}

In this section we prove this main lemma.

\begin{lemma}
 Suppose that   $\meso^\gamma\in \mathcal H(\vec\lambda)$ and 
that\label{notsplitting} 
  $\mathcal Y$ is  a set
of fewer than $\lambda$ nice 
$\meso^\gamma$-names of subsets of $\omega$,
then there is a $\delta<\gamma+\lambda$
and an
extension $\meso^{\delta}$ of $\meso^\gamma$
 in $\mathcal H(\vec\lambda)$
that forces that the family $\mathcal Y$ is not  a splitting family.
\end{lemma}

The main theorem follows easily.

\bgroup

 \def\proofname{Proof of Theorem \ref{mainth}.}

 \begin{proof}
   Let $\theta$ be any regular cardinal so
    that  $\theta^{<\lambda} = \theta $ (for example, $\theta =
    (2^\lambda)^+$).    Construct $\meso^\theta
\in\mathcal H(\vec\lambda)$ so that for all
   $\mathcal Y\subset 
   \mathbb B_{\kappa,\theta}$ with $|\mathcal Y|<\lambda$,
   there is a $\gamma<\delta<\theta$ so that
    $\mathcal Y\subset \mathbb B_{\kappa,\gamma}$ and,
by applying Lemma \ref{notsplitting}, 
such that $\meso^\theta\restriction\delta$ forces that
$\mathcal Y$ is not  a splitting family.
 \end{proof}

\egroup

We begin by reducing our job to simply 
finding  a $(\meso, \vec\lambda(i_\gamma))$-thin sequence.

\begin{definition}
  For  a $(\kappa,\gamma)$-matrix-iteration $\meso^\gamma$,
  we say that\label{setlambdathin} a subset $\mathcal E$
of $\mathbb B_{\kappa,\gamma}$ is
$(\meso^\gamma,\vec\lambda(i_\gamma))$-thin 
filter subbase
if, $i_\gamma<\kappa$, 
$|\mathcal E|\leq \mu_{i_\gamma}$, and
 the sequence $\langle \langle \mathcal E\cap \mathbb B_{i,\gamma}
 \rangle : i< \kappa\rangle $ is a
$(\meso^\gamma,\vec\lambda(i_\gamma))$-thin sequence of filter bases. 
\end{definition}

\begin{lemma}
  For  any  $\meso^\gamma
\in\mathcal H(\vec\lambda)$,
and any $(\meso^\gamma,\vec\lambda(i_\gamma))$-thin filter base
$\mathcal E$, there is an $\alpha \leq \gamma+\mu_{i_\gamma}+1$
and  extensions  $\meso^{\alpha},\meso^{\alpha+1}$ of $\meso^\gamma$
in $\mathcal H(\vec\lambda)$, such that,\label{nodirection} 
$\meso^{\alpha+1} = \meso^{\alpha} *
 \mathbb L(\langle \dot{\mathfrak D}^\alpha_i : i_\alpha\leq i <
 \kappa\rangle)$  and 
$\meso^{\alpha}$ forces that
$\mathcal E\cap \mathbb B_{i,\gamma}$ is a subset
of $ \dot{\mathfrak D}^\alpha_i$ for all $i <\kappa$.
\end{lemma}

\begin{proof}
The case $i_\gamma = 0$ is trivial, so we assume 
 $i_\gamma>0$.
There is no loss of generality to assume that
 $\mathcal E\cap \mathbb B_{i_\gamma,\gamma}$ has character
 $\mu_{i_\gamma}$. Let $\{ \dot E_\xi : \xi < \mu_{i_\gamma}\}
\subset \mathcal E\cap \mathbb B_{i_\gamma,\gamma}$
 enumerate a filter base for $
\langle\mathcal E\rangle\cap \mathbb B_{i_\gamma,\gamma}$. We can
assume that this enumeration satisfies 
that $\dot E_\xi \setminus \dot E_{\xi+1}$ is forced to be infinite
for all $\xi < \mu_{i_\gamma}$. 
Let $\mathcal A$ be any countably generated  free filter on
$\omega$ that is not principal mod finite. 
By induction on $\xi <\mu_{i_\gamma}$ we define
 $\meso^{\gamma+\xi}$ by simply defining  $i_{\gamma+\xi}$ and
the sequence 
 $\langle \dot {\mathfrak D}^{\gamma+\xi}_i : i_{\gamma+\xi} \leq
 i\leq\kappa\rangle$. 
We will also recursively define, for each $\xi<\mu_{i_\gamma}$,
 a $\meso^{\gamma+\xi}$-name $\dot D_\xi$ such that
 $\meso^{\gamma+\xi}$ forces that $\dot D_\xi\subset \dot E_\xi$. 
An important induction hypothesis is that
 $\{ \dot D_\eta : \eta < \xi \} \cup \{ \dot E_\zeta : \zeta
 <\mu_{i_\gamma}\}\cup \mathcal E$ is forced to have the finite
 intersection 
 property. 

For each $\xi <
\gamma+\omega_1$,
 let $i_{\xi} = 0$ and $\dot{\mathfrak D}^\xi_i $
be the $\meso^\xi$-name $
\langle \mathcal A\rangle\cap \mathbb B_{i,\xi}$
 for all $i \leq\kappa$. 
The definition of $\dot D_0$ is simply $\dot E_0$. By recursion, for
each $\eta < \omega_1$ and $\xi = \eta+1$, we define 
 $\dot D_\xi $ to be the intersection of $\dot D_\eta$ and $\dot
 E_\xi$. For limit $\xi<\omega_1$, we note that
 $\mathbb P_{i_\gamma,\xi}$ forces
that $\mathbb L(\langle \mathcal A\rangle)$ is isomorphic to
 $\mathbb L(\langle \{ \dot D_\eta\cap \dot E_\xi : \eta < \xi
 \}\rangle )$. Therefore, we can let $\dot D_\xi$ be a 
 $\meso^{\xi+1}$-name for the generic real added by
 $\mathbb L(\langle \{ \dot D_\eta\cap \dot E_\xi : \eta < \xi
\}\rangle)$. A routine density argument shows that this definition
satisfies the induction hypothesis. 

The definition of $i_{\gamma+\omega_1}$ is $i_\gamma$ and the
definition of $\dot {\mathfrak D}^{\gamma+\omega_1}_{i_\gamma}$ is 
the filter generated by $\{ \dot D_\xi : \xi <\omega_1\}$. 
The definition of $\dot D_{\omega_1}$ is $\dot L_{\gamma+\omega_1}$.

Let $S$ denote the set of $\eta<\mu_{i_\gamma}$ with uncountable
cofinality. 
We now add  additional  induction hypotheses:
\begin{enumerate}
\item if $\zeta=\sup(S\cap \xi)<\xi$ and $\xi = \nu+1$, 
    then $\dot D_\xi = \dot D_\nu\cap \dot E_\xi$, and
 $i_\xi=0$ and $\dot{ \mathfrak D}^{\gamma+\xi}_i = \langle \mathcal
 A\rangle $ for all $i\leq \kappa$
\item if $\zeta=\sup(S\cap \xi)<\xi$ and $\xi $ is a limit of
  countable cofinality, then 
$i_\xi=0$ and $\dot{ \mathfrak D}^{\gamma+\xi}_i = \langle \mathcal
 A\rangle $ for all $i\leq \kappa$, and
$\dot D_\xi$ is forced
 by $\meso^{\gamma+\xi+1}$ to be 
the generic real added by $\mathbb L( \{ \dot D_\eta \cap \dot E_\xi : 
   \zeta\leq \eta<\xi\})$,
\item if $\zeta = \sup(S\cap \xi)$ and $\xi=\zeta+\omega_1$, 
  then $i_\xi = i_\gamma$, 
 $\dot {\mathfrak D}^{\gamma+\xi}_{i_\xi} $ is the filter generated by 
   $\{ \dot E_\xi \cap \dot D_\eta : \zeta\leq \eta < \xi\}$
and $\dot D_\xi$ is $\dot L_{\gamma+\xi}$,
\item if $S\cap \xi$ is cofinal in $\xi$ and $\cf(\xi)>\omega$,
 then $i_\xi = i_\gamma$ and
        $\dot {\mathfrak D}^{\gamma+\xi}_{i_\xi}$ is the filter
        generated by $\{ \dot D_{\gamma+\eta}  : \eta\in S\cap \xi\}$
 and $\dot D_\xi = \dot L_{\gamma+\xi}$,
\item if $S\cap \xi$ is cofinal in $\xi$ and $\cf(\xi)=\omega$,
 then $i_\xi = 0$  and
$\dot{ \mathfrak D}^{\gamma+\xi}_i = \langle \mathcal
 A\rangle $ for all $i\leq \kappa$, 
and 
$\dot D_\xi$ is forced  by $\meso^{\gamma+\xi+1}$ to be 
the generic real added by 
$\mathbb L( \{ \dot D_{\eta_n} \cap \dot E_\xi : 
n\in \omega\})$, where $
\{\eta_n : n\in \omega\}$ is some increasing cofinal 
subset of $S\cap (\gamma,\xi)$.
\end{enumerate}

It should be clear that the induction continues to stage
 $\mu_{i_\gamma}$ and that
 $\meso^{\gamma+\xi}\in \mathcal
 H(\vec\lambda(i_\gamma))$ for all $\xi\leq \mu_{i_\gamma}$,
 with $\beta_{\gamma_\xi} = \gamma$ being the witness to
Definition \ref{H} for all $\xi$ with $\cf(\xi)>\omega$.

 The final definition of the sequence
 $\langle \dot{\mathfrak D}^{\delta}_{i} : i_\delta = i_\gamma  \leq
 i \leq \kappa\rangle$, where $\delta = \gamma+\mu_{i_\gamma}$ is
that
 $\dot{\mathfrak D}^\delta_{i_\gamma} $ is the filter generated by
 $\{ \dot L_{\gamma+\xi} : \cf(\xi)>\omega\}$, and for $i_\gamma <
 i\leq \kappa$, $\dot {\mathfrak D}^\delta_i$ is the filter generated
 by $\dot {\mathfrak D}^\delta_{i_\gamma}\cup (\mathcal E\cap 
\mathbb B_{i,\gamma}$.  
\end{proof}

\begin{lemma} 
Suppose that 
$\mathcal E$ is a\label{addsmall}
 $(\meso^\gamma,\vec\lambda(i_\gamma))$-thin filter base.
Also assume that $i<\kappa$ and $\alpha\leq \gamma$
and $\mathcal E_1\subset \mathbb B_{i,\alpha}$  is a 
 $(\meso^\alpha,\vec\lambda(i_\gamma))$-thin filter base
satisfying that  $\langle\mathcal E\rangle
\cap \mathbb  B_{i,\alpha}\subset \langle\mathcal E_1\rangle$, 
then $\mathcal E\cup \mathcal E_1$ is
a $(\meso^\gamma,\vec\lambda(i_\gamma))$-thin filter subbase.
\end{lemma}

\begin{proof}
Let $\mathcal E_2$ be equal to $\mathcal E\cup \mathcal E_1$. 
The fact that each member of the sequence
 $\langle \dot{\mathcal D}_j = 
\langle \mathcal E_2\cap \mathbb B_{j,\gamma}\rangle :
 j<\kappa\rangle $ is a name of a filter base with character at most
 $\mu_{i_\gamma}$ is immediate. 
Now we verify that if $j_1 < j_2 < \kappa$, then 
 $\Vdash_{\mathbb P_{j_2,\gamma} } \dot{\mathcal D}_{j_2}\cap \mathbb
 B_{j_1,\gamma} \subset \dot{\mathcal D}_{j_1}$.
Let $\dot b\in \mathbb B_{j_2,\gamma}$ and suppose there are
 $p\in \mathbb P_{j_2,\gamma}$, $\dot E_0\in \mathcal E\cap \mathbb
 B_{j_2,\gamma}$, and $\dot E_1 \in \mathcal E_1$ such that
 $p\Vdash b\cap \dot E_0 \cap \dot E_1$. It suffices to produce
an $\dot E\in \langle\mathcal E_2\rangle
\cap \mathbb B_{j_1,\gamma}$ satisfying that
 $p\Vdash \dot b\cap \dot E=\emptyset$. First, using that $\mathcal E$
 is $(\meso^\gamma,\vec\lambda(i_\gamma))$-thin, choose $\dot E_2\in
\langle \mathcal E\rangle\cap \mathbb B_{j_1,\gamma}$ such that 
 $p\Vdash (\dot b\setminus \dot E_0) \cap \dot E_2
 =\emptyset$. Equivalently, we have that 
 $p\Vdash (\dot b\cap \dot E_2)\subset \dot E_0$, and therefore
$p\Vdash (\dot b\cap \dot E_2)\cap \dot E_1 =\emptyset$. Since $\dot
E_1$ is a $\mathbb P_{j_2,\alpha}$-name, there is a $\mathbb
P_{j_1,\alpha}$-name (which we can denote as)
 $(\dot b\cap \dot E_2)\restriction\alpha$ satisfying that
 $p\Vdash 
 \dot E_2\cap (\dot b\cap \dot E_2)\restriction\alpha $ is empty
and that 
$p\Vdash 
(\dot b\cap \dot E_2)\subset
(\dot b\cap \dot E_2)\restriction\alpha$. Now using  that
$\mathcal E_1$ is $(\meso^\alpha,\vec\lambda(i_\gamma))$-thin,
 choose $\dot E_3 \in \langle \mathcal E_1\rangle \cap \mathbb
 B_{j_1,\alpha}$ so that
$p\Vdash \dot E_3 \cap (\dot b\cap \dot E_2)\restriction\alpha$ is
empty. Naturaly we have that
$p\Vdash \dot E_3 \cap (\dot b\cap \dot E_2)$ is also empty.
This completes the proof since $\dot E_2\cap \dot E_3$ is
in $\langle \mathcal E_2\rangle \cap \mathbb B_{j_1,\gamma}$.
\end{proof}

Let $\meso^\gamma\in \mathcal H(\vec\lambda)$ and let
 $\dot y\in \mathbb B_{\kappa,\gamma}$. 
For a family  $\mathcal E\subset \mathbb B_{\kappa,\gamma}$
and condition  $p\in \meso^\gamma$ say that $p$
forces that $\mathcal E$
measures $\dot y$ if $p\Vdash_{\meso^\gamma} \{\dot y,\omega\setminus
\dot y\}\cap \langle \mathcal E\rangle \neq\emptyset$. 
Naturally we will just say that $\mathcal E$
 measures $\dot y$ if
$1 $ forces that $\mathcal E$ measures $\dot y$.

Given Lemma \ref{nodirection}, it will now suffice to prove.

\begin{lemma}
If $\mathcal Y\subset \mathbb B_{\kappa,\gamma}$ for some
$\meso^\gamma\in \mathcal H(\vec \lambda)$ and $|\mathcal Y|\leq
\mu_{i_\gamma}$ for some $i_\gamma<\kappa$, then there is
a $(\meso^\gamma,\vec\lambda(i_\gamma))$-thin filter\label{existsthin}
$\mathcal E\subset \mathbb B_{\kappa,\gamma}$ that measures every
element of $\mathcal Y$. 
\end{lemma}

In fact,  to prove Lemma \ref{existsthin}, it is evidently sufficient
to prove:

\begin{lemma}
If $\meso^\gamma\in \mathcal H(\vec \lambda)$,
  $\dot y\in \mathbb B_{\kappa,\gamma}$,
and if $\mathcal E$ is
a $(\meso^\gamma,\vec\lambda(i_\gamma))$-thin filter,
 then
 there is a family $\mathcal E_1\supset \mathcal E$
measuring $\dot y$ 
that is also 
a $(\meso^\gamma,\vec\lambda(i_\gamma))$-thin filter.
\end{lemma}

\begin{proof}
Throughout the proof
 we suppress mention of
$\meso^\gamma$ and refer instead to
component member posets $ 
\mathbb P_{i,\alpha}, \dot {\mathbb Q}_{i,\alpha}$
 of $\meso^\gamma$. 
Let $i_{\dot y}$ be minimal such
that $\dot y$ is in $\mathbb B_{i_{\dot y},\gamma}$.
Proceeding by induction, we can assume that the lemma holds  for all
$\dot x\in  \mathbb B_{j,\gamma}$ and all $j<i_{\dot y}$.

We can replace $\dot y$ by any $\dot x\in \mathbb B_{i_{\dot
    y},\gamma}$ that has the property that
 $1 \Vdash \dot x\in \{\dot y, \omega\setminus \dot y\}$ since
if we measure  $\dot x$ then we also measure $\dot y$. With this
reduction then we can assume that no condition forces
that $\omega\setminus \dot y$ is in the filter generated by 
 $\mathcal E$.

\begin{fact} If $i_{\dot y}\leq i_\gamma$, then
there is a $\dot E\in \mathbb B_{i_{\dot y},\gamma}$
such that  
$\mathcal E \cup \{\dot E\}$ is contained 
 a $(\meso^\gamma,\vec\lambda(i_\gamma))$-thin filter
that measures $\dot y$.
\end{fact}

\bgroup

\def\proofname{Proof of Fact 1}

\begin{proof}
It is immediate that 
 $\langle \{\dot y\} \cup (\mathbb B_{i_{\dot y},\gamma}\cap 
\mathcal E)\rangle $ is 
a  $(\meso^\gamma,\vec\lambda(i_\gamma))$-thin filter. 
Therefore, by Lemma \ref{addsmall}, $\mathcal E\cup \{\dot y\}$ 
is a $(\meso^{\gamma},\vec\lambda(i_{\gamma}))$-thin filter subbase.
\end{proof}

\egroup

We may thus assume that $0<i_{\dot y}$ and that
 the Lemma has been proven for all
members of $\mathbb B_{i,\gamma}$ for all $i< i_{\dot y}$.
Similarly, 
 let $\alpha_{\dot y}$ be minimal so 
that $\dot y\in \mathbb B_{i_{\dot y},\alpha_{\dot y}}$,
and assume that the Lemma has been proven for all members of $\mathbb
B_{i_{\dot y},\beta}$ for all $\beta < \alpha_{\dot y}$. 
We skip proving the easy case when $\alpha_{\dot y}=1$ and
henceforth assume that $1<\alpha_{\dot y}$. 
Notice also that $\alpha_{\dot y}$ has countable cofinality
since $\mathbb P_{i_{\dot y},\gamma}$ is ccc.

Now choose an elementary submodel $M$ 
  of $H((2^{\lambda\cdot\gamma})^+) $
  containing $\vec\lambda, \meso^\gamma,\mathcal E,\dot y$  and
  so that  $M$ has 
  cardinality equal to  $\mu_{i_\gamma}$
  and, by our cardinal assumptions,
 $M^{\lambda_j} \subset M$ 
for all $j<i_\gamma$.
Naturally this implies that $M^\omega\subset M$.

By the inductive assumption we may assume that there is
an $\mathcal E_1\supset \mathcal E$ that is
$(\meso^\gamma,\vec\lambda(i_\gamma))$-thin and measures every
element of $M\cap \mathbb B_{j,\gamma}$ for $j<i_{\dot y}$
as well as every element of $M\cap \mathbb B_{i_{\dot y},\beta}$
for all $\beta \in M\cap \alpha_{\dot y}$. 
Moreover, it is easily checked that we can assume that
$\mathcal E_1$ is a subset of $M$. Furthermore,
we may assume that $\mathcal E_1$ contains  a maximal
family of subsets of $M\cap \mathbb B_{i_{\dot y}, \alpha_{\dot y}}$
that forms a  $(\meso^\gamma,\vec\lambda(i_\gamma))$-thin 
filter subbase.

\begin{fact} There is a maximal antichain $A\subset 
 \mathbb P_{i_{\dot y},\gamma}$  and a subset  $A_1\subset A$ 
such that 
\begin{enumerate}
\item each $p\in A_1$ forces that $\mathcal E_1$ measures $\dot y$,
\item  for each $p\in A\setminus A_1$,
 $p$ forces that
there is an $i_p < i_{\dot y}$ such that
$\mathbb B_{i_p,\gamma}\cap \langle \mathcal E_1\cup \{\dot y\}\rangle$ 
 is not generated by the elements in $M$,
\item  for each $p\in A\setminus A_1$,
 $p$ forces that
there is a $j_p < i_{\dot y}$ such that $i_p\leq j_p$
and
$\mathbb B_{j_p,\gamma}\cap \langle \mathcal E_1\cup \{\omega\setminus  
\dot y\}\rangle$ is not generated by the elements in $M$.
\end{enumerate}
\end{fact}

\bgroup
\def\proofname{Proof of Fact 2}
\begin{proof}
Suppose that $p\in \mathbb P_{i_{\dot y},\gamma}$ forces
 that the conclusion  (2) fails.
We have already arranged that
 $p\Vdash_{\mathbb P_{i_{\dot y},\gamma}} 
 \dot y \in \langle \mathcal E_1\cap \mathbb B_{i_{\dot y},\gamma}
 \rangle^+$.  
Define $\dot E\in \mathbb B_{i_{\dot y},\gamma}$
so that $p$ forces $\dot E = \dot y$ and each $q\in 
\mathbb P_{i_{\dot y},\gamma}\cap {p}^\perp$ 
forces that $\dot E = \omega$. It is easily checked
that $\mathbb B_{i_{\dot y},\gamma}\cap \langle
\mathcal E_1\cup \{\dot E\}\rangle$ is then 
$(\meso^\gamma,\vec\lambda(i_\gamma))$-thin and that
 $p$ forces that it measures $\dot y$.
This condition ensures that $p$ is compatible with an
element of  $ A_1$.

If (2) holds but (3) fails, then by a symmetric argument as in
the previous paragraph we can again define  $\dot E$ so 
that $\mathbb B_{i_{\dot y},\gamma}\cap \langle
\mathcal E_1\cup \{\dot E\}\rangle$ is then 
$(\meso^\gamma,\vec\lambda(i_\gamma))$-thin and that
 $p$ forces that it measures $\omega\setminus \dot y$.
\end{proof}
\egroup

If by increasing $M$ we can enlarge $A_1$ we simply do so. Since
$\meso^\gamma$ is ccc we may assume that this is no longer possible,
and therefore we may also assume that $A$ is a subset of $M$.
Now we choose any $p\in A\setminus A_1$. It suffices to produce
an $\dot E_p\in \mathbb B_{i_{\dot y},\gamma}$ that can be added
to $\mathcal E_1$ that measures $\dot y$ and satisfies
that $q\Vdash \dot E_p = \omega$ for all $q\in p^\perp$. This is
because we then have that $\mathcal E_1 \cup \{ \dot E_p : p\in
A\setminus A_1\}$ is contained in a $\vec\lambda(i_\gamma)$-thin
filter that measures $\dot y$.

\begin{fact} There is an $\alpha$ such that $\alpha_{\dot y}  = \alpha
  +1$.
\end{fact}

\bgroup

\def\proofname{Proof of Fact 3}

\begin{proof}

Otherwise, let $j=i_p$ and
for each $r<p$ in $\mathbb P_{i_{\dot y},\alpha_{\dot y}}$,
choose $\beta \in M\cap \alpha_{\dot y}$ such that $r\in 
\mathbb P_{i_{\dot y},\beta}$,
 and define a name
$\dot y[r]$  in $M\cap \mathbb B_{j,\gamma}$
 according to $(\ell,q)\in \dot y[r]$
providing there is a pair 
$(\ell,p_\ell)\in \dot y$ such that 
 $q<_j p_\ell$ and
$q\restriction\beta $ is in the set
$M \cap \mathbb P_{j,\beta}
\setminus (r\wedge p_\ell\restriction\beta)^\perp $  . 
This set, namely $\dot y[r]$,
 is in $M$ because $\mathbb P_{j,\beta}$ is ccc 
and $M^\omega\subset M$.

We prove that
 $r$ forces that $\dot y[r]$ contains $\dot y$.
 Suppose that $r_1 < r$ and there is a pair
 $(\ell,p_\ell)\in \dot y$ with $r_1 < p_\ell$. 
Choose an $r_2\in \mathbb P_{j,\gamma}$ so that $ r_2 <_j
r_1$. It suffices to show $r_2\Vdash \ell\in
 \dot y[r]$.
Let $q<_j p_\ell$ with $q\in M$.   Then $r_2\not\perp p_\ell$ implies
 $r_2\not\perp q$. Since $r_2$ was any $<_j$-projection of
 $r_1$ we can assume that $r_2<q$. Since $r_2\restriction\beta$ is
in $(\mathbb P_{j,\beta}\cap (r\wedge
p_\ell\restriction\beta)^\perp)^\perp$, it follows 
that $q\restriction \beta \notin (r\wedge
p_\ell\restriction\beta)^\perp$. This implies that $(\ell,q)\in \dot
y[r]$ and completes the proof that $r_2\Vdash \ell \in \dot y[r]$.

Now assume that $\beta<\alpha_{\dot y}$ and
$r \Vdash \dot b \cap \dot E \cap \dot y$
is empty for some $r<p$ in $\mathbb P_{i_{\dot y},\beta}$,
 $\dot b\in \mathbb B_{j,\gamma}$, and
 $\dot E\in \mathcal E_1\cap \mathbb B_{i_{\dot y},\gamma}$. 
Let $\dot x = (\dot E\cap \dot y)[r]$ (defined as above
for $\dot y[r]$).  We complete the proof of Fact 3 by proving that
 $r\Vdash \dot b \cap \dot x$ is empty. Since each are
 in $\mathbb B_{j,\gamma}$, we may choose any $r_1 <_j r$,
 and assume that $r_1 \Vdash \ell \in \dot b\cap \dot x$. In addition
 we can suppose that there is a pair $(\ell,q)\in \dot x$ such
that $r_1 < q$. The fact that $(\ell,q)\in \dot x$ means
there is a $p_\ell$ with $(\ell, p_\ell) $ in the name
 $\dot E\cap \dot y$ such that $q<_j p_\ell$. 
Since $r_1\in \mathbb P_{j,\gamma}$ and $r_1 <q$, 
 it follows that $r_1\not \perp p_\ell$. Now it follows
that $r_1$ has an extension forcing that $\ell\in \dot b\cap (\dot
E\cap \dot y)$ which is a contradiction.
\end{proof}

\begin{fact} $i_{\dot y} = i_\alpha$ and so\label{ivalue} also $i_p<i_\alpha$.
\end{fact}

\def\proofname{Proof of Fact \ref{ivalue}}

\begin{proof}
Since $\mathbb P_{i,\alpha+1} = \mathbb P_{i,\alpha}$ for
$i<i_\alpha$, we  have that $i_\alpha \leq i_{\dot y}$.
Now assume that $i_\alpha <i_{\dot y} $ and we proceed much as we did
in Fact 3 to prove that $i_p$ does not exist.
 Assume that $r<p$ (in
 $\mathbb P_{i_{\dot y},\alpha+1}$ and $r\Vdash \dot b\cap (\dot E\cap
 \dot y)$ is empty for some $\dot E\in M\cap \langle \mathcal E_1\rangle
\cap \mathbb
 B_{i_{\dot y},\gamma}$ and $\dot b\in \mathbb B_{i_p,\gamma}$. 
It follows from Lemma \ref{addsmall} that we can simply assume
that  $\dot E\in \mathcal E_1\cap \mathbb
 B_{i_{\dot y},\alpha+1}$, and similarly that
 $\dot b\in \mathbb B_{i_p,\alpha+1}$.

 Let $\dot T_{\alpha}$ be the $\mathbb P_{i_{\dot
    y},\alpha}$-name such that $r\restriction \alpha \Vdash
 r(\alpha) = \dot T_\alpha \in \mathbb L(\mathfrak D^\alpha_{i_{\dot
     y}})$. We may assume that there is a $t_\alpha \in
 \omega^{<\omega}$ such that $r\restriction\alpha \Vdash 
 t_\alpha = \sakne(\dot T_\alpha)$.

Choose any $M\cap \mathbb P_{i_\alpha,\alpha}$-generic filter $\bar G$
such that $r\restriction\alpha \in \bar G^+$. Since $\mathbb
P_{i_\alpha,\alpha}$ is ccc and $M^\omega\subset M$, it follows
that $M[\bar G]$ is closed
under $\omega$-sequences in the model
 $V[\bar G]$. 

In this model, define an $\mathbb L(\mathfrak D^\alpha_{i_\alpha})$-name
 $\dot x$. A pair $(\ell,T_\ell)\in \dot x$ if
 $t_\alpha\leq\sakne(T_\ell)\in T_\ell\in 
\mathbb L(\mathfrak D^\alpha_{i_\alpha})$ and for each $\sakne(T_\ell)\leq t\in
 T_\ell$, there is a pair $(\ell,q_{\ell,t})\in M $ in the name
$ (\dot y\cap \dot E)$ such that $q_{\ell,t}\restriction\alpha\in \bar
G^+$, $q_{\ell,t}\restriction\alpha\Vdash t =
\sakne(q_{\ell,t}(\alpha))$, 
and $(q_{\ell,t}\restriction\alpha \wedge r\restriction \alpha)$ does not
force (over the poset $ \bar G^+$~) that $t\notin \dot
T_\alpha$. We will show that   $r$ forces over
the poset $\bar G^+$ that $\dot x$ contains $\dot E\cap \dot
y$
and that $\dot x\cap \dot b$ is empty. This proves that $p$
forces that $\langle\mathcal E_1\rangle
\cap \mathbb B_{i_p,\alpha+1}$ generates
$\langle\mathcal E_1\cup\{\dot y\}\rangle
 \cap \mathbb B_{i_p,\alpha+1}$
since $\dot x$ must be forced to be in $\langle\mathcal E_1\rangle$.
 It then follows from Lemma \ref{addsmall}
 that  $\mathcal E_1\cap \mathbb B_{i_p,\gamma}$ generates
$\langle \mathcal E_1\cup\{\dot y\}\rangle \cap \mathbb
B_{i_p,\gamma}$,
 contradicting the assumption on $i_p$.

To prove that $r$ forces that $\dot x$ contains $\dot y\cap \dot E$, we
consider any  $r_\ell < r$ that forces over $ \bar G^+$
that $\ell\in 
\dot y\cap  \dot E
$. We may choose $(\ell,p_\ell) \in M$ in the name
$ (\dot E\cap \dot y)$ such that (wlog)
$r_\ell<p_\ell$. We may assume that $r_\ell\restriction \alpha$
forces a value $t$ on $\sakne(r_\ell(\alpha)) $
and that this equals $\sakne(p_\ell(\alpha))$. Now show there
is a $T_\ell\in \mathbb L(\mathfrak D^\alpha_{i_\alpha})$. 
In fact,  assume $t\in T_\ell$ with $q_{\ell,t}$ as the witness.
 Let $L^- = \{ k : t^\frown k \notin T_\ell\}$;  it suffices to show
 that $L^- \notin (\mathfrak D^\alpha_{i_\alpha})^+$.

By assumption
that $q_{t,\ell}$ is the witness, there is an $r_t <
(q_{\ell,t}\restriction \alpha \wedge r\restriction \alpha)$ such
that $r_t \Vdash t\in \dot T_\alpha$
and $r_t\Vdash t = \sakne(q_{\ell,t}(\alpha))$. 
By strengthening $r_t$ we can assume that $r_t$ forces a value
 $\dot D\in \dot{\mathcal D}^\alpha_{i_{\dot y}}$ on 
 $\{ k : t^\frown k \in \dot T_\alpha\cap q_{\ell,t}(\alpha)\}$. 
But now, it follows that $r_t$ forces that $\dot D$
is disjoint from $L^-$ since if $r_{t,k}\Vdash k\in \dot D$ 
for some $r_{t,k}<r_t$, $r_{t,k}$ is the witness to 
$(\ell,q_{\ell,t^\frown k})$ is in $(\dot y\cap \dot E)$ etc.,
 where $q_{\ell,t^\frown k}\restriction \alpha =
 q_{\ell,t}\restriction\alpha$ and $
q_{\ell,t^\frown k}( \alpha ) =  (q_{\ell,t}(\alpha))_{t^\frown k}$.
Since some condition forces that $L^-$ is not in 
 $(\dot {\mathfrak D}^\alpha_{i_{\dot y}})^+$ it follows
that $L^-$ is not in 
 $(\dot {\mathfrak D}^\alpha_{i_\alpha})^+$

Finally we must  show that $r$ forces over $\bar G^+$
that $\dot b$ is disjoint from  $\dot x$.
Since each are $\mathbb P_{i_p,\alpha+1}$-names, it suffices
to assume  that 
 $\bar r\in \bar G^+$  is some $\mathbb P_{i_p,\alpha+1}$-reduct of
 $r$ that forces some $\ell$ is in $\dot b\cap \dot x$,
and to 
 then show that 
$r$ fails to force that $\ell\notin \dot b\cap (\dot E\cap \dot y)$.  
Choose $(\ell,q_{\ell,t})\in (\dot y\cap \dot E)$ witnessing that
 $\bar r\Vdash \ell\in \dot x$. That is, we may assume that
 $\bar r \restriction \alpha \Vdash t =\sakne(\bar r(\alpha))$,
that  $q_{\ell,t}\restriction \alpha \in \bar G^+$, 
 and $(q_{\ell,t}\wedge r\restriction\alpha)$ does not force
over $\bar G^+$ that $t\notin \dot T_\alpha$. Of course this means
that the condition $\bar r \wedge r \wedge [[t\in \dot
T_\alpha]]\wedge q_{\ell,t}$ is not $0$. This condition forces
that $\ell$ is in $\dot b\cap (\dot E\cap \dot y)$ as required.
\end{proof}

\egroup

\begin{fact} The character\label{character}
 of $\mathfrak D^\alpha_{i_\alpha}$ is
  greater than $\mu_{i_\gamma}$.
\end{fact}

\bgroup
\def\proofname{Proof of Fact \ref{character}}

\begin{proof}

We know that $\mathfrak D^{\alpha}_{i_\alpha}$ is forced
to have an $\omega$-closed base (in fact, descending mod finite with
uncountable cofinality). Even more, $\mathbb P_{i_\alpha,\alpha}$
forces
that for all $T\in \mathbb L(\mathfrak D^\alpha_{i_\alpha})$, there is a
$D\in \mathfrak  D^\alpha_{i_\alpha}$ such that the condition
$([D]^{<\omega})_{\sakne(T)}$ is below $ T$.  Let $\chi_\alpha$ be the
cofinality of $\alpha$ and
fix a list $\{ \dot D_\beta : \beta < \chi_\alpha\}\in M$ (closed
under mod finite changes)
of 
$\mathbb
P_{i_\alpha,\alpha}$-names of elements of $\dot{\mathfrak
  D}^\alpha_{i_\alpha}$ that is forced to be a base.

Now, suppose that $\dot b\in
\mathbb B_{i_p,\alpha+1} = \mathbb B_{i_p,\alpha}$ and
 there is an $\dot E\in \mathcal E_1$ and an $r<p$ forcing
that $\dot b\cap (\dot E\cap \dot y)$ is empty.  We prove there is an 
$\dot x\in \mathcal E_1$ and an $r_2<r\restriction \alpha$ in 
 $\mathbb P_{i_\alpha,\alpha}$ such that $r_2 \Vdash \dot b \cap \dot
 x$ is empty. We may assume that $r_2$ forces a value $t$ 
on $\sakne(r(\alpha))$ and that, for some $\beta <\chi_\alpha$,
 $r_2\Vdash (\dot D_\beta^{<\omega})_t < r(\alpha)$. 
Let \[\dot x = \{ (\ell, q_\ell\restriction\alpha) : 
 (\ell,q_\ell)\in (\dot E\cap \dot y)~~\mbox{and}\ \ 
 q_\ell\restriction\alpha \Vdash q_\ell(\alpha) \leq (\dot
 D_\beta^{<\omega})_t \}~.\]
It is immediate that
 $\dot x\in M$ and that
 $(r_2\wedge r)
\Vdash_{\mathbb P_{i_\alpha,\alpha+1}} \dot x\supseteq (\dot E\cap
\dot y)$.  Since $\dot E\cap \dot y$ is forced to be 
 in $\mathcal E_1^+$, it follows that $\dot x$ is forced
by $r_2$ 
 to be
 in $\langle\mathcal E_1\rangle$.  Now we verify that
 $r_2\Vdash \dot b\cap \dot x$ is empty. Assume that
 $r_3<r_2$ in $\mathbb P_{i_\alpha,\alpha}$ and that $r_3\Vdash
 \ell\in \dot b\cap \dot x$. We may assume there
is $(\ell,q_\ell\restriction\alpha)\in \dot x$ such that $r_3 
 < q_\ell\restriction\alpha$. But now $r_2 \Vdash q_\ell(\alpha) \leq
 r(\alpha)$ and so $r_2\wedge r\Vdash \ell \in \dot b\cap (\dot E\cap
 \dot y)$ -- a contradiction.

The conclusion now follows from Lemma \ref{addsmall}.
\end{proof}

\egroup

\begin{definition} For each $t\in \omega^{<\omega}$, define that 
 $\mathbb P_{i_\alpha,\alpha}$-name $\dot E_t$ according to the rule
that $r\Vdash \ell\in \dot E_t$ providing $r\in 
 \mathbb P_{i_\alpha,\alpha}$ forces that
there is a $\dot T$ with\label{Es} 
 $r\Vdash \dot T\in \mathbb L(\dot{\mathfrak D}^\alpha_{i_\alpha})$,
 $r\Vdash t=\sakne(\dot T)$, and $r \cup \{(\alpha,\dot T)\}\Vdash
 \ell\notin \dot y$.
\end{definition}

\begin{fact} There is a $\dot T\in \mathbb L(\dot{\mathfrak
    D}^\alpha_{i_\alpha})\cap M$ such that $p\restriction \alpha$
forces\label{EsinE} the statement:\\
{$\dot E_t\in \mathcal E_1$ for all $t$ such that
 $\sakne(\dot T)\leq t\in \dot T$~.}
\end{fact}

\bgroup
\def\proofname{Proof of Fact \ref{EsinE}}

\begin{proof}
By elementarity, there is a maximal antichain of 
 $\mathbb P_{i_\alpha,\alpha}$ each element of which decides
if there is a $\dot T$ with $\dot E_t\in \mathcal E_1$ for
 all $t\in \dot T$ above $\sakne(\dot T)$. 
Since $p\in A\setminus A_1$ it follows that 
there is an $i_p<i_\alpha$ as in  condition (2) of
Fact 2.  Let $t_0\in \omega^{<\omega}$ so that
$p\restriction\alpha\Vdash t_0=\sakne(p(\alpha))$. 
 By the maximum principle, there is
 a $\dot b\in \mathbb B_{i_p,\gamma}$ and a $\dot E_0\in \mathcal E_1$ 
 satisfying that
 $p\Vdash \dot b\cap \dot E_0\cap  \dot y$ is empty,
 while $p\Vdash \dot b\cap \dot E$ is infinite for all $\dot E\in
 \langle\mathcal E_1\rangle$. This means
that $p$
forces that $\dot b\cap \dot E_0$ is an element of 
 $\langle\mathcal E_1\rangle^+$ that is contained in $\omega\setminus
 \dot y$.  
As in the proof of Lemma \ref{addsmall}, there is an $\dot E_2\in 
\langle \mathcal E_1\rangle \cap \mathbb B_{i_p,\gamma}$ such that
 $p$ forces that $\dot b\cap \dot E_2$ is contained in $\dot E_0$. 
We also have that $(\dot b\cap \dot E_2)\restriction \alpha$ is forced
to be contained in $\omega\setminus \dot y$. It now follows that 
$p\restriction\alpha$ forces that for
all $t_0\leq t\in p(\alpha)$, $p\restriction \alpha $ forces
that $\dot E_t $ contains $(\dot b\cap \dot E_2)\restriction\alpha$
and so  is in $\langle\mathcal E_1\rangle^+$. Since $\dot E_t$ is
also measured by $\mathcal E_1$, we have that $p\restriction\alpha$
forces that such $\dot E_t$ are in $\mathcal E_1$. This completes the
proof.  
\end{proof}

\egroup

\bigskip

Now we show how to extend $\mathcal E_1\cap \mathbb B_{i_\alpha,\gamma}$ so
as to measure $\dot y$. Let $\beta = \sup(M\cap \alpha)$. By Fact
\ref{character}, $\beta <\alpha$ and by the definition of 
 $\mathcal H(\vec\lambda)$, $M\cap \dot{\mathfrak D}^\alpha_{i_\alpha}
 $ is a subset of $\langle \dot{\mathfrak D}^\beta_{i_\beta}\rangle$,
 $\dot L_\beta\in \dot{\mathfrak D}^\alpha_{i_\alpha}$, and
 $i_\beta = i_\alpha$. We also have that the family
 $\{ \dot L_\xi : \cf(\xi)\geq\omega_1 \ \mbox{and}\ 
\beta_\alpha \leq \xi \in M\cap \beta\}$ is a base
 for
 $\dot {\mathfrak D}^\beta_{i_\beta}$.  For convenience
 let $q<_M p$ denote the relation that $q$
 is an $M\cap \mathbb P_{i_\alpha,\alpha+1}$-reduct of $p$.
Let $\bar p$ be any condition in $ \mathbb P_{i_\beta,\beta+1}$
satisfying that  $\bar p\restriction \beta = p\restriction \alpha$
and $\bar p\restriction \beta \Vdash \sakne(\bar p(\beta)) =
 t_\alpha$ (recall that $p\restriction\alpha\Vdash t_\alpha = 
 \sakne(p(\alpha)$).

Let us note that for each $q\in M\cap \mathbb P_{\alpha,i_\alpha+1}$,
 $q\restriction\alpha = q\restriction\beta$
and $q\restriction \beta\Vdash q(\alpha)
$ is also  a 
$\mathbb P_{\beta,i_\beta}$-name of an element of $\mathbb
 L(\dot{\mathfrak D}^\beta_{i_\beta})$. 
 Let $\dot x$ be the following $\mathbb
 P_{i_\beta,\beta+1}$-name 
\[\dot x  = \{ (\ell,q\restriction \beta \cup \{(\beta,q(\beta))\} ): 
   (\ell, q)\in \dot y\cap M \ \mbox{and} \  q <_M p\}~.\]

We will complete the  proof by showing  that there is an extension of 
 $p$  that forces that 
 $\mathcal E_1\cup \{\omega\setminus (\dot x[\dot L_\beta])\}$
measures $\dot y$ and that $1$ forces that 
 $\langle\mathcal E_1\cup \{\omega\setminus (\dot x[\dot
 L_\beta])\}\rangle
 \cap \mathbb B_{i_{\dot y},\beta+1}$
is $\vec \lambda(i_\gamma)$-thin. Here
 $\dot x[\dot L_\beta]$ abbreviates the $\mathbb P_{i_\beta,\beta+1}$-name
\[ \{ (\ell, r) : (\exists q)~~ (\ell,q)\in \dot x, \
q\restriction\beta = r\restriction \beta, 
 \  \mbox{and} \ r
\Vdash \sakne(q(\beta))\in \dot L_\beta^{<\omega}
\}.
\]
The way to think of $\dot x[\dot L_\beta]$ is that if $\bar p
 $ is  in some $\mathbb P_{i_\alpha,\alpha}$-generic filter $G$,
 then $\dot y[G]$ is now an $\mathbb L(\mathfrak
 D^\alpha_{i_\alpha})$-name, $L_\beta^{<\omega} = (\dot
 L_\beta[G])^{<\omega}$ is 
in $\mathbb L(\mathfrak D^\alpha_{i_\alpha})$,
 and $(\dot x[\dot L_\beta])[G]$ 
is equal to $\{ \ell :  L_\beta^{<\omega} \not\Vdash \ell\notin \dot
y\}$. We will use the properties of $\dot x$ to help show
that $\mathcal E_1\cup \{ \omega\setminus (\dot x[\dot L_\beta])\}$ is
$\vec \lambda(i_\gamma)$-thin. This semantic description of $\dot
x[\dot L_\beta]$ makes clear that $\bar p \cup \{(\alpha,
 (\dot L_\beta)^{<\omega})\} \in \mathbb P_{i_\alpha,\alpha+1}$ forces
that $\dot x[\dot L_\beta]$ contains $\dot y$. This implies that
$\mathcal E_1\cup \{ \omega\setminus (\dot x[\dot L_\beta])\}$  
measures $\dot y$.

Claim:  It is forced by $\bar p$ that $\omega\setminus \dot x$ is not 
measured by $\mathcal E_1$.

Each element  of $\mathcal E_1$ is in $M$ and simple elementarity will
show that for any condition in $q$ in $M$ that forces 
 $\dot E\cap (\omega\setminus \dot y)$ is infinite,
 the corresponding $\bar q = q\restriction \alpha
\cup \{(\beta, q(\alpha))\}$  will also force
that 
 $\dot E\cap (\omega\setminus \dot x)$ is infinite.

It follows from Fact \ref{character}, with $\omega\setminus \dot x$
playing the role of $\dot y$, that $\mathcal E_1\cup 
\{\omega\setminus \dot x\}$ is $\vec\lambda(i_\gamma)$-thin. Recall
that $q\Vdash \dot x = \emptyset$ for all $q\perp \bar p$.
Now to prove that $\mathcal E_1\cup \{\omega\setminus (\dot x[\dot
L_\beta])\} $ is also $\vec \lambda(i_\gamma)$-thin, we prove that 
\[\langle \mathcal E_1\cup \{\omega\setminus \dot x\}\rangle \cap
\mathbb B_{i,\alpha} 
 = 
\langle \mathcal E_1\cup \{\omega\setminus (\dot x[\dot
L_\beta])\}\rangle \cap \mathbb B_{i,\alpha} \]
for all $i< i_\alpha$. 
In fact, first we prove 
\[\langle \mathcal E_1\cup \{\omega\setminus \dot x\}\rangle \cap
\mathbb B_{i,\beta} 
 = 
\langle \mathcal E_1\cup \{\omega\setminus (\dot x[\dot
L_\beta])\}\rangle \cap \mathbb B_{i,\beta} \]
for all $i< i_\alpha$.

We begin with this main Claim.

\begin{claim}  If  $\dot b\in \mathbb B_{i,\beta}$ ($i<i_\beta$)
and there is an $\dot E\in \mathcal E_1\cap\mathbb
B_{i_\alpha,\beta}$ and a 
 $\bar p \geq q\in \mathbb P_{i_\beta,\beta+1}$ such that 
  $q \Vdash \dot b \cap  (\dot E\setminus \dot x)= \emptyset$
then 
 $q\restriction \beta \Vdash (\exists \dot E\in \mathcal E_1)~~
 \dot b \cap \dot E=\emptyset$.
\end{claim}

\bgroup
\def\proofname{Proof of Claim:}
\begin{proof}
We may assume that $q\restriction \beta$ forces a value
 $t$ on $\sakne(q(\beta))$. 
Recall that $q\restriction \beta$ forces the statement:
there is a $\dot D
\in M\cap \dot {\mathfrak D}^\alpha_{i_\alpha}$ such that 
$(\dot D^{<\omega})_{t}\leq q(\beta)$. The definition of 
 $\dot x$ ensures that $q\restriction \beta\cup \{(\alpha, 
 (\dot D^{<\omega})_t )\}\Vdash \dot b \cap (\dot E\setminus \dot y)$ is
 empty. There is a $\mathbb P_{i_\alpha,\alpha}$-name $\dot E_1\in
 M$ such
 that
 $q\restriction \alpha\Vdash \dot E_1 = \{ \ell :
   ~~(\dot D^{<\omega})_t \not\Vdash \ell\notin (\dot E\setminus \dot
   y)\}$. By assumption $q\restriction \alpha \Vdash \dot E_1 \in
   \langle\mathcal E_1\rangle$. Since $\dot b$ is also a $\mathbb
   P_{i,\alpha}$-name, we have that $q\restriction \alpha \Vdash \dot
   b\cap \dot E_1 = \emptyset$. 
\end{proof}
\egroup

Now assume that $\dot b\in \mathbb B_{i_\beta,\beta}$ and 
$q\Vdash \dot b\cap (\dot E\cap (\omega\setminus (\dot x[\dot
L_\beta])))$ is empty for some $q<\bar p$ in $\mathbb
P_{i_\beta,\beta+1}$. By Lemma \ref{addsmall} it suffices
to assume that $\dot E\in \mathbb B_{i_\beta,\beta}$.
To prove that $q$ forces that $
\dot b\notin\langle \mathcal E_1\rangle^+$, it suffices
to prove that there is some $\dot E_1\in\mathcal E_1$ such
that 
$q\Vdash \dot b\cap (\dot E_1\cap (\omega\setminus \dot x))$ is
finite. We proceed by contradiction.

 We may again assume that $q\restriction\beta$ forces that $q(\beta)$
 is 
 $(\dot D^{<\omega})_t$ for some $t\supset t_\alpha$
and some $\dot D\in \dot{\mathfrak D}^\alpha_{i_\alpha}\cap M$. 
Let $H$ be the range of $t$.
Let, for the moment, $G$ be a $\mathbb P_{i_\alpha,\alpha}$-generic
filter with $q\in G$.  Now in $M[G]$ we have the value
$L_\beta$ of $\dot L_\beta$ and $H\subset L_\beta$. We can also let 
 $E$ denote the value of $\dot E[G]$. Recall
that for each $s\in H^{<\omega}$,
 $E_s $ denotes the set of $\ell\in E$ such that
there is some $T\in \mathbb L(\mathfrak
 D^\alpha_{i_\alpha})$ with $s =  \sakne(T)$ and 
 $T\Vdash \ell\notin \dot y$.  
We have shown in Fact \ref{EsinE} that
there is a $T\in \mathbb L(\mathfrak D^\alpha_{i_\alpha})\cap M$ such
that $E_s\in \mathcal E_1$ for all $s\in T$ above $\sakne(T)$. 
This means that there is an $\ell\in b\cap E$ such that
$\ell\in E_s$ for each of the finitely many suitable $s$.
For each $s$, choose $T_s\subset T$ witnessing $\ell\in E_s$.
As before, and since there are only finitely many $s$ involved,
 we can assume that $\dot T_s = (\dot D^{<\omega})_s$ for
some   $H\subset \dot D\in \dot{\mathfrak D}^\alpha_{i_\alpha}\cap M$
and we then
define an extension $q$ of $q$ so that
$q'(\beta) = ( \dot D^{<\omega})_{t_\alpha}$
ensures
that $(\dot L_\beta^{<\omega})_s< T_s$ for each $s$.
Note that
such a condition $q'$
we have that $q'\cup \{(\alpha,(\dot L_\beta)^<\omega)\}
$ forces that   $\ell\notin \dot y$. 
But
then it should be clear that
$q'$ that forces $\ell\notin \dot x[\dot L_\beta]$.  This contradicts
that $q$ forces $\ell\notin \dot b\cap (\dot E\cap (\omega
\setminus(\dot x[\dot L_\beta])))$.

\end{proof}

\end{document}